\documentclass[english,a4paper]{amsart}
\usepackage[T1]{fontenc}
\usepackage{ae}
\usepackage{aecompl}
\usepackage{amssymb}
\usepackage{mathrsfs}

\input{xy}
\xyoption{all}
\SelectTips{cm}{}
\usepackage{babel}
\usepackage[pdftex]{graphicx}
\usepackage[leqno]{amsmath}
\usepackage{amsthm}
\usepackage[svgnames]{xcolor}

\renewcommand\phi{\varphi}
\renewcommand\epsilon{\varepsilon}
\renewcommand\theta{\vartheta}

\newcommand\mbb{\mathbb}

\newcommand\mcal{\mathscr}

\newcommand\ul{\underline}

\newcommand{\V}{{\mathcal V}}
\def\TH{\ensuremath{\textup{TH}}}

\newcommand\sA{\mcal{A}}

\newcommand\sC{\mcal{C}}

\newcommand\sL{\mcal{L}}

\newcommand\sS{\mcal{S}}

\newcommand\N{\mbb{N}}

\newcommand\R{\mbb{R}}

 \newcommand{\la}[0]{\lambda}
 \newcommand{\si}[0]{\sigma}

\DeclareMathOperator\Tr{Tr}

\DeclareMathOperator\pr{pr}

\DeclareMathOperator\QM{\rm QM}

\DeclareMathOperator\Sym{Sym}

\DeclareMathOperator\conv{conv}
\DeclareMathOperator\M{M}
\DeclareMathOperator\cc{cc}

\DeclareMathOperator\id{id}

\DeclareMathOperator\interior{\rm int}

\numberwithin{equation}{section}

\theoremstyle{plain}
\newtheorem{Thm}[equation]{Theorem}
\newtheorem{Prop}[equation]{Proposition}
\newtheorem{Cor}[equation]{Corollary}
\newtheorem{Lemma}[equation]{Lemma}

\newtheorem*{Thm*}{Theorem}
\newtheorem*{Satz*}{Satz}
\newtheorem*{Prop*}{Proposition}
\newtheorem*{Cor*}{Corollary}
\newtheorem*{Lemma*}{Lemma}
\newtheorem*{Hilfssatz*}{Lemma}
\newtheorem*{Sublemma*}{Sublemma}
\newtheorem*{Conjecture*}{Conjecture}

\theoremstyle{definition}
\newtheorem{Def}[equation]{Definition}

\newtheorem{Ex}[equation]{Example}

\newtheorem{Question}[equation]{Question}
\newtheorem{Rem}[equation]{Remark}

\newtheorem*{Def*}{Definition}
\newtheorem*{Defs*}{Definitions}
\newtheorem*{Ex*}{Example}
\newtheorem*{Exs*}{Examples}
\newtheorem*{LemmaDef*}{Lemma and Definition}
\newtheorem*{Notation*}{Notation}
\newtheorem*{Problem*}{Problem}
\newtheorem*{Question*}{Question}
\newtheorem*{Rem*}{Remark}
\newtheorem*{Rems*}{Remarks}
\newtheorem*{Warning*}{Warning}

\begin{document}
\title{Positive Polynomials and Projections of Spectrahedra}
\author{Jo\~{a}o Gouveia }
\address{Department of Mathematics, University of Washington, Seattle, USA, and CMUC, Department of Mathematics, University of 
Coimbra, Portugal}
\email{jgouveia@math.washington.edu}
\thanks{The first author was partially supported by the NSF Focused 
   Research Group grants DMS-0757371 and DMS-0757207 and by Funda{\c c}{\~ a}o para a Ci{\^ e}ncia e Tecnologia.}
 
\date{\today}
\author{Tim Netzer}
\address{Falkult\"at f\"ur Mathematik und Informatik, Universit{\"a}t Leipzig, Germany}
\email{tim.netzer@math.uni-leipzig.de}

\begin{abstract} This work is concerned with different aspects of spectrahedra and their projections, sets that are important in semidefinite optimization. We prove results on the limitations of so called Lasserre and theta body relaxation methods for semialgebraic sets and varieties.  As a special case we obtain the main result of  \cite{NePlSch} on non-exposed faces. We also solve the open problems from that work. We further prove some helpful facts which can not be found in the existing literature, for example that the closure of a projection of a spectrahedron is again such a projection. We give a unified account of several results on convex hulls of curves and images of polynomial maps. We finally prove a Positivstellensatz for projections of spectrahedra, which exceeds the known results that only work for basic closed semialgebraic sets.
\end{abstract}

\maketitle
\section{Introduction} \textit{Semidefinite programming} has turned out to be a very important and valuable tool in polynomial optimization in recent times. It is concerned with finding optimal values of linear functions on certain convex sets. These sets, called \textit{spectrahedra}, arise as linear sections of the cone of positive semidefinite matrices. Semidefinite programming generalizes linear programming. The importance of semidefinite programming comes from two facts. On one hand there exist efficient algorithms to solve semidefinite programming problems, see for example Ben-Tal and Nemirovski \cite{MR1857264}, Nesterov and Nemirovski \cite{MR1258086}, Nemirovski \cite{MR2334199}, Vandenberghe and Boyd \cite{MR1379041} and Wolkowicz, Saigal and Vandenberghe \cite{MR1778223}.
On the other hand, a  great amount of problems from various branches of mathematics can be approached using semidefinite programming. Examples come from combinatorial optimization, non-convex optimization and control theory; see for example Parrilo and Sturmfels \cite{MR1995016}, Gouveia, Parrilo and Thomas \cite{GPT} and all of the above mentioned literature.

This brings up the theoretical question of how to characterize sets on which semidefinite programming can be performed, i.e. to characterize spectrahedra. Helton and Vinnikov \cite{MR2292953} have done groundbreaking work towards this question. They show that spectrahedra are what they call \textit{rigidly convex}, and this condition is sufficient in dimension two. This result also solves the Lax conjecture, as explained in Lewis, Parrilo and Ramana \cite{MR2146191}. Whether each rigidly convex set of higher dimension is a spectrahedron is still an open question. However, it was just recently shown by Br\"and\'en \cite{bran} that the higher dimensional Lax conjecture fails.

Observe that semidefinite programming can also be performed on projections of spectrahedra. One just has to optimize the objective function over a higher dimensional set. Up to now there are only two known necessary conditions for a set to be the projection of a spectrahedron: being convex and being semi-algebraic. Lasserre \cite{LasserreConvSets} has provided a method to prove for certain sets that they are the projection of a spectrahedron.  Helton and Nie \cite{HeltonNieSDPrepr,HeltonNieNecSuffSDP} have applied the method to large classes of convex sets. They indeed conjecture that each convex semi-algebraic set is the projection of a spectrahedron. 

This work is concerned with the question of how to write sets as projections of spectrahedra. Our contribution is the following. After introducing notation we review in Section 3 some of the methods to construct projections of spectrahedra. We give  a unified account of some results spread across  the literature, for example on convex hulls of curves, and we prove some helpful facts for which we could not find a reference. For example we show that the closure of the projection of a spectrahedron is again such a projection.

In Section 4 we analyze the Lasserre method, and the related theta body method from \cite{GPT}. We prove results on the limitations of these approaches. As a special case we obtain the main result from \cite{NePlSch}.  We also settle the open questions from that work by providing a series of examples.

Finally we proof a Positivstellensatz for projections of spectrahedra in Section 5. This is interesting in particular because such sets are usually not basic closed semialgebraic. So none from   the large amount of known Positivstellens\"atze apply to such sets.

\section{Notation}\label{not}

We will use the following notation. For $n\in\N$ let $\ul X =(X_1,\ldots,X_n)$ be an $n$-tuple of variables. Let $\R[\ul X]$ denote the real polynomial ring in these variables. By $\R[\ul X]_d$ we denote its finite dimensional subspace of polynomials of degree at most $d$. Let $\ul p=(p_1,\ldots,p_r)$ be an $r$-tuple of polynomials from $\R[\ul X]$. Then $$\sS(\ul p):=\{x\in\R^n\mid p_1(x)\geq 0, \ldots, p_r(x)\geq 0\}\subseteq\R^n$$ is the \textit{basic closed semi-algebraic set} defined by $\ul p$. In the polynomial ring we have a corresponding \textit{quadratic module}, defined as $$\QM(\ul p):=\left\{ \si_0 + \si_1p_1+\cdots +\si_rp_r\mid \si_i\in\sum \R[\ul X]^2\right\}.$$ Here we use the notation $\sum V^2$ for the set of all sums of squares of elements from a given subset $V$ of a commutative ring $R$. 

All elements from $\QM(\ul p)$ are nonnegative as functions on $\sS(\ul p)$.
There are also certain truncated parts of $\QM(\ul p)$, defined as $$\QM(\ul p)_d:=\left\{ \si_0+\si_1p_1+\cdots+\si_rp_r\mid \si_i\in\sum(\R[\ul X]_d)^2\right\}.$$ $\QM(\ul p)_d$ is contained in the finite dimensional space $\R[\ul X]_{2d+\nu}$, where $\nu$ is the maximum over the degrees of $p_1,\ldots,p_r$. Note however that $\QM(\ul p)_d$ will be strictly smaller than $\QM(\ul p)\cap\R[\ul X]_{2d+\nu}$ in general.

We denote by $\M_{k\times k}(V)$ the set of $k\times k$-matrices with entries from a given subset $V$ of a commutative ring $R$. $\sum M_{k\times k}(V)^2$ is then the set of \textit{sums of hermitian squares}, i.e. it contains the finite sums of elements of the form $A^tA$ with $A\in\M_{k\times k}(V)$. We denote by $\Sym_k(V)$ the set of symmetric matrices from $\M_{k\times k}(V)$. The usual inner product $A\circ B$ for $k\times k$-matrices $A=(a_{ij})_{i,j}$ and $B=(b_{ij})_{i,j}$ is defined as $$A\circ B=\Tr(AB)=\sum_{i,j}a_{ij}b_{ij},$$ where $\Tr$ denotes the trace. For a matrix $A\in\Sym_k(\R),$ $A\succeq 0$ means that $A$ is positive semidefinite, i.e. $v^tAv\geq 0$ holds for every $v\in\R^k$. $A\succ 0$ means that $A$ is positive definite, i.e. $v^tAv>0$ holds for all $v\neq 0$.

A $k$-dimensional \textit{linear matrix polynomial} is an affine linear polynomial $$\sA(\ul X)= A + X_1 B_1+\cdots X_n B_n,$$ with $A,B_1,\ldots,B_n\in\Sym_k(\R).$ It is called \textit{strictly feasible} if there is a point $x\in\R^n$ with $\sA(x)\succ 0$.
 The set $$\sS(\sA):=\{x\in\R^n\mid \sA(x)\succeq 0\}$$ is called a \textit{spectrahedron}. It is a convex and basic closed semi-algebraic set, and a generalization of a polyhedron. This paper deals with projections of such spectrahedra, i.e. sets of the form $$S=\{x\in\R^n\mid \exists y\in\R^m\ \sA(x,y)\succeq 0\},$$ where $\sA$ is a linear matrix polynomial in the variables $X_1,\ldots,X_n,Y_1,\ldots,Y_m$. So $S$ is the image of the spectrahedron $\widetilde{S}\subseteq\R^{n+m}$ defined by $\sA$, under the canonical projection $\R^{n+m}\rightarrow\R^n$.
 
 For a convex set $S\subseteq\R^n$ let $\mathrm{Aff}(S)$ denote its affine hull, i.e. the smallest affine subspace of $\R^n$ containing $S$. A \textit{face} of $S$ is a nonempty convex subset $F\subseteq S$ which is extremal in the following sense: whenever $\la x+(1-\la)y\in F$ for some $x,y\in S,\la\in(0,1)$, then $x,y\in F$. For an affine linear polynomial $\ell\in\R[\ul X]_1$ that is nonnegative on $S$, the subset $\{x\in S\mid \ell(x)=0\}$ is a face or empty. A face if called \textit{exposed} if it is of such a form.

\section{Some construction methods revisited}

As indicated in the introduction, there is a large amount of works on  the construction of spectrahedra that project to a given set. In this section we review some of them. We also provide 
proofs of some helpful facts that can not be found in the existing literature. 
\subsection{Polars and Closures}\label{first}
 We start by reviewing a result on polars by Nemirovski, and we deduce some helpful corollaries. We for example observe that the closure of the projection of a spectrahedron is again such a projection. The results on polars will also be very helpful in the subsequent section, when considering Lasserre relaxations.

In \cite{MR2334199}, Section 4.1.1, Nemirovski proves the following result: 

\begin{Prop} \label{nem}Let $\sA(\ul X,\ul Y)= A + X_1B_1 +\cdots + X_nB_n + Y_1C_1+ \cdots + Y_mC_m$ be a $k$-dimensional strictly feasible linear matrix polynomial. Let $S:=\{x\in\R^n\mid \exists y\in\R^m\ \sA(x,y)\succeq 0\}$ be the projection of the spectrahedron defined by $\sA,$ and let $$S^{\circ}:=\{\ell\in\R[\ul X]_1 \mid \ell\geq 0 \mbox{ on } S\}$$ denote the convex cone of affine linear polynomials nonnegative on  $S$. Then \begin{align*}S^{\circ}=\{ l_0+l_1X_1+\cdots+l_nX_n \mid \exists U\in\Sym_k(\R): \quad & U\succeq 0,\  U\circ A\leq l_0, \\ & U\circ B_i=l_i \mbox{ for } i=1,\ldots, n, \\ & U\circ C_j=0 \mbox{ for } j=1,\ldots, m \}.\end{align*} In particular, $S^{\circ}$ is again the projection of a spectrahedron.
\end{Prop}
The result follows from the duality theory of conic programming, and is thus essentially a separation argument. The set  $S^{\circ}$ is called the \textit{polar} of $S$ in Nemirovksi's work.  

In a first step we want to get rid of the technical assumption \textit{strictly feasible} in Proposition \ref{nem}.

\begin{Lemma}\label{inter} Let $S\subseteq\R^n$ be the projection of a spectrahedron $\widetilde{T}\subseteq\R^{n+l}$ and assume $\interior(S)\neq \emptyset.$ Then $S$ is the projection of a spectrahedron $\widetilde{S}\subseteq\R^{n+m}$ with $\interior(\widetilde{S})\neq \emptyset$ and $m\leq l$. \end{Lemma}
\begin{proof} Let $\widetilde{T}\subseteq\R^{n+l}=\R^n\times\R^l$ be a spectrahedron that projects to $S$. Let $e_i^{(n)}$ denote the $i$-th standard basis vector of $\R^n$. Without loss of generality we assume $0, e_1^{(n)},\ldots, e_n^{(n)}\in S$ (which uses $\interior(S)\neq\emptyset$). So we have $b_i:=(e_i^{(n)},u_i)\in\widetilde{T}$ for some $u_i\in\R^l$ ($i=1,\ldots,n)$. We can also assume $0\in \widetilde{T}$, so if $V$ denotes the affine hull of $\widetilde{T}$, then $V$ is a subspace of $\R^n\times\R^l$. Note that $\widetilde{T}$ has nonempty interior in $V$. We extend $b_1,\ldots, b_n$ to a basis of $V$, by adding some $c_1,\ldots, c_m$. Then we extend these vectors to a basis of $\R^n\times\R^l$ by adding vectors $d_1,\ldots, d_t$ (so $m+t=l$).  We can thereby choose all $c_i,d_i\in\{0\}^n\times\R^l$.  

Now let $L\colon \R^n\times\R^l\rightarrow\R^n\times\R^l$ be the linear automorphism sending $b_i$ to $e_i^{(n+l)},$ $c_i$ to $e_{n+i}^{(n+l)}$ and $d_i$ to $e_{n+m+i}^{(n+l)}$.

Note that for all $x\in\R^n$ and $u\in\R^l$ there is some $\widetilde{u}\in\R^l$ with $L(x,u)=(x,\widetilde{u})$. We further have $L(V)=\R^n\times\R^m\times\{0\}^t$. Now \begin{align*} S=&\left\{ x\in\R^n\mid \exists u\in\R^l\ (x,u)\in \widetilde{T}\right\} \\ =& \left\{ x\in\R^n\mid \exists u\in\R^l \ L(x,u)\in L(\widetilde{T})\right\} \\ =& \left\{ x\in\R^n\mid \exists v\in \R^m \ (x,v,0)\in L(\widetilde{T})\right\}. \end{align*} Since $L(\widetilde{T})$ is clearly also a spectrahedron, and considering it as a spectrahedron $\widetilde{S}$ in $\R^n\times\R^m$, we have proven the result.
\end{proof}

Note that for a spectrahedron, having nonempty interior is equivalent to being definable by a strictly feasible linear matrix polynomial, by Ramana and Goldman \cite{MR1342934}, Corollary 5.
So we get:

\begin{Prop} \label{pol}Let $S\subseteq\R^n$ be the projection of a spectrahedron. Then $S^{\circ}=\{\ell\in\R[\ul X|_1\mid \ell\geq 0 \mbox{ on } S\}$ is again the projection of a spectrahedron.
\end{Prop}
\begin{proof} First assume that $S$ has nonempty interior in $\R^n$. Then $S$ is the projection of a spectrahedron defined by as strictly feasible linear matrix polynomial, by Lemma \ref{inter} and Corollary 5 in \cite{MR1342934}. So the result follows from Theorem \ref{nem} in this case.

If $S$ has empty interior, assume without loss of generality that its affine hull is $\R^t\times\{0\}^{n-t}.$ Then $S$ has nonempty interior considered as a set in $\R^t$. If $S^{\circ}_t$ denotes the polar of $S$ in $\R[X_1,\ldots, X_t]$, then $$S^{\circ}=\{\ell\in\R[\ul X]_1\mid \ell(X_1,\ldots,X_t,0,\ldots,0)\in S^{\circ}_t \},$$ which proves the result.
\end{proof}

\begin{Cor} \label{clos} Let $S\subseteq\R^n$ be the projection of a spectrahedron. Then its closure $\overline{S}$ is again the projection of a spectrahedron.
\end{Cor}
\begin{proof} By Corollary \ref{pol}, $(S^{\circ})^{\circ}$ is the projection of a spectrahedron. But we have $$\overline{S}=\{x\in\R^n\mid X_0 +x_1X_1+\cdots + x_nX_n\in (S^{\circ})^{\circ} \}, $$ which proves the result. 
\end{proof}

We can also use Proposition \ref{nem} for an alternative characterization of projections of spectrahedra:

\begin{Cor} For a closed convex set $S\subseteq \R^n$, the following are equivalent:
\begin{itemize}
\item[(i)] $S$ is the projection of a spectrahedron.
\item[(ii)] $S$ is the inverse image under an affine linear map of the dual of a spectrahedral cone.
\end{itemize}
\end{Cor}
\begin{proof} "(ii)$\Rightarrow$(i)" follows from Proposition \ref{pol}. For "(i)$\Rightarrow$(ii)" first assume that $S$ has nonempty interior. Then it is the projection of a spectrahedron defined by a strictly feasible $k$-dimensional linear matrix polynomial $$\sA(\ul X,\ul Y)=A+X_1B_1+\cdots+X_nB_n+Y_1C_1+\cdots+Y_mC_m.$$ Since $S$ is closed we find by Proposition \ref{nem}  \begin{align*}S=\{x\in\R^n\mid U\circ\sA(x,0)\geq 0 & \mbox{ for all } U\in\Sym_k(\R) \mbox{ with } \\ & U\succeq 0 \mbox{ and }  U\circ C_j=0 \mbox{ for all } j\}.\end{align*} So if  $\sC$ denotes the spectrahedral cone of positive semidefinite matrices $U$ fulfilling the linear equations $U\circ C_j=0$ for all $j$, then $S$ is the inverse image of the dual of $\sC$ under the affine linear map $x\mapsto \sA(x,0)$. 

Now assume without loss of generality that $S\subseteq\R^t\times\{0\}^{n-t}$ has nonempty interior in $\R^t$. Then there is some affine linear map $L\colon\R^t\rightarrow\R^s$ and a spectrahedral cone $\sC\subseteq\R^s$ such that $S=L^{-1}(\sC^{\vee})$. Here, $\sC^\vee$ denotes the dual cone of $\sC$ in $\R^s$. Then for the spectrahedral cone $\widetilde{\sC}:= \sC \times\R^{n-t}$ and affine linear map $$\widetilde{L}\colon\R^n\rightarrow\R^s\times\R^{n-t};\  (x,y)\mapsto(L(x),y)$$ one as $\widetilde{L}^{-1}(\widetilde{\sC}^{\vee})=S$.
\end{proof}

\subsection{Lasserre Relaxations}\label{second}

In this subsection we review the method of Lasserre \cite{LasserreConvSets} to construct projections of spectrahedra, and use Proposition \ref{pol} to give an alternative explanation of the method.

 We first observe that  if $M\subseteq\R[\ul X]_1$ is the projection of a spectrahedron, then $$\sL:=\{x\in\R^n\mid \ell(x)\geq 0 \mbox{ for all } \ell \in M\}$$ is also such a projection. This follows from Proposition \ref{pol}, since $\sL$ is $M^{\circ}$ intersected with a subspace.  Now for a finite set of polynomials $p_1,\ldots,p_r\in\R[\ul X]$ let $S=\sS(\ul p)$ be the basic closed semi-algebraic set they define, $\QM(\ul p)$ denote the corresponding quadratic module in $\R[\ul X]$ and $\QM(\ul p)_d$ its truncated part, as defined in Section \ref{not}. It turns out that each $\QM(\ul p)_d$ is the projection of a spectrahedron. One can for example use the following result, which is Theorem 1 from Ramana and Goldman  \cite{Ramana95quadraticmaps}:

\begin{Thm} \label{rago}Let $f\colon\R^n\rightarrow \R^m$ be a quadratic polynomial map. Then the convex hull of the image $f(\R^n)$ is the projection of a spectrahedron.
\end{Thm}
So note that each $\QM(\ul p)_d$ is the convex hull of the image of a quadratic map. Indeed one just has to parametrize the coefficients occuring in the sums of squares used in the representations of its elements. Thus the sets $\QM(\ul p)_d\cap\R[\ul X]_1$ are projections of spectrahedra and
we finally obtain that each set $$\sL(\ul p)_d:=\{x\in\R^n\mid \ell(x)\geq 0 \mbox{ for all } \ell \in \QM(\ul p)_d\cap\R[\ul X]_1\}$$ is the projection of a spectrahedron. The set $\sL(\ul p)_d$ is called a degree $d$ Lasserre relaxation of $S$. Each $\sL(\ul p)_d$ is closed convex and contains $S$. The sequence of the $\sL(\ul p)_d$ is descending. 
 
 Note that our definition of a Lasserre relaxation differs slightly from the original one given in  Lasserre  \cite{LasserreConvSets}. There, the dual of $\QM(\ul p)_d$ is projected to $\R^n,$ whereas we intersect $\QM(\ul p)_d$ with $\R[\ul X]_1$ and then pass to $\R^n$. However, the different definitions define the same relaxations \textit{up to closures}, at least if $S$ has a nonempty interior. This can for example be checked with an argument as in the proof of Proposition 3.1 in Netzer, Plaumann and Schweighofer \cite{NePlSch}, using the closedness of $\QM(\ul p)_d$. 
 
 The following Theorem is the key result on Lasserre relaxations. Part (i) is mainly Theorem 2 from Lasserre \cite{LasserreConvSets}, and now clear from our above considerations. Part (ii) is proven as Proposition 3.1 (2) in Netzer, Plaumann and Schweighofer \cite{NePlSch}.

\begin{Thm}\label{lasmain}(i) If  $\QM(\ul p)_d$ contains all affine linear polynomials nonnegative on $S$, then $\sL(\ul p)_d=\overline{\conv(S)}$. In particular, $\overline{\conv(S)}$ is the projection of a spectrahedron then. 

(ii) If $\sL_d(\ul p)=\overline{\conv(S)}$ and $S$ has nonempty interior, then $\QM(\ul p)_d$ contains all affine linear polynomials nonnegative on $S$.
 \end{Thm}

Another possibility for obtaining semidefinite descriptions for convex sets is a different Lasserre-type relaxation hierarchy for 
convex hulls of algebraic sets, the {\it theta body} hierarchy introduced in Gouveia, Parrilo and Thomas \cite{GPT}.
Given an ideal $I\subseteq\R[\ul X]$, we denote the set of all polynomials $p$ such that $p-\si \in I$ for some sum of squares $\si$ with $\deg(\si) \leq 2d$ by $\Sigma(d,I)$. Note that $\Sigma(d,I)$ intersected with a finite dimensional subspace of $\R[\ul X]$ is the projection of a spectrahedron. This follows since $I\cap W$ is a subspace in $W$, for each subspace $W$ of $\R[\ul X]$.

\begin{Def}
Let $I\subseteq \R[\ul X]$ be an ideal. The $d$-th theta body of $I$, 
denoted by $\TH(I)_d$, is the intersection of all half-spaces $H_{\ell}:=\{x\in\R^n\mid  \ell(x) \geq 0\},$
where $\ell$ ranges over all linear polynomials in $\Sigma(d,I)$. 
\end{Def}
 The theta body hierarchy for the ideal $I$ approximates the convex hull of its real variety
 $\V_{\R}(I)=\{x\in\R^n\mid g(x)=0 \mbox{ for all } g \in I\}$.  
An analogous result to Theorem \ref{lasmain} is  true, with the condition 
of the ideal $I$ being real radical replacing the condition of $S$
having nonempty interior.

\begin{Thm}\label{thm:thetamain}(i)  If  $\Sigma(d,I)$ contains all affine linear polynomials nonnegative on $\V_{\R}(I)$, then $\TH(I)_d=\overline{\conv(\V_{\R}(I))}$. In particular, $\overline{\conv(\V_{\R}(I))}$ is the projection of a spectrahedron then. 

(ii) If $\TH(I)_d=\overline{\conv(\V_{\R}(I))}$ and $I$ is real radical, then $\Sigma(d,I)$ contains all affine linear polynomials 
nonnegative on $\V_{\R}(I)$.
 \end{Thm}
Again, part (i) is immediate from the definition, while part (ii) is proven in Lemma 2.7 of \cite{GPT}. In Section \ref{ex} we will study possible obstructions to these two methods. In particular we reprove the main result of Netzer, Plaumann and Schweighofer \cite{NePlSch} and settle the open problems from that work. 

\subsection{Images of Polynomial Maps}\label{fourth}

In this subsection we want to give a unified account of several results on convex hulls of images under polynomial maps, including results by Lasserre, Parrilo, Ramana and Goldman, Henrion and Scheiderer. The results can all be deduced  from the following principle:

\begin{Prop}\label{trans} Let $S\subseteq\R^n$ be a set and $V\subseteq\R[\ul X]$ a finite dimensional linear subspace containing $1$. Assume the subset $P\subseteq V$ of all elements of $V$ that are nonnegative on $S$ is the projection of a spectrahedron. Then for any map $f=(f_1,\ldots, f_m)\colon\R^n\rightarrow\R^m$  with $f_i\in V$ for all $i,$  $$\overline{\conv(f(S))}\subseteq\R^m$$ is the projection of a spectrahedron.
\end{Prop}
\begin{proof} For any affine linear polynomial $\ell\in\R[Y_1,\ldots,Y_m]_1$ the  polynomial \linebreak $\ell(f_1,\ldots,f_m)$ belongs to $V$. Define $M:=\{\ell\in\R[Y_1,\ldots,Y_m]_1\mid \ell(f_1,\ldots,f_m)\in P\}.$ One immediately checks that $M$ is the projection of a spectrahedron (since $P$ is) and contains only polynomials that are nonnegative on $f(S)$. Conversely, if $\ell$ is affine linear and nonnegative on $f(S)$, then $\ell(f_1,\ldots,f_m)$ is in $P$. Thus $M$ is precisely the cone of affine linear polynomials nonnegative on $f(S)$, and by the arguments from the last section $$\overline{\conv(f(S))}=\{x\in\R^m\mid \ell(x)\geq 0 \mbox{ for all } \ell \in M\}$$ is the projection of a spectrahedron.
\end{proof}

\begin{Ex} Not very surprisingly, the Lasserre result can be recovered from Proposition \ref{trans}. Indeed if there is some $d$ such that $\QM(\ul p)_d$ contains all affine linear polynomials that are nonnegative on $S$, then apply Proposition \ref{trans} with $V=\R[\ul X]_1$ and $f=\id$. $P=V\cap\QM(\ul p)_d$ is the projection of a spectrahedron, as explained in the previous section.
\end{Ex}

\begin{Ex} We also get that the closure of $\conv(f(\R^n))$ is the projection of a spectrahedron, for any quadratic map $f\colon\R^n\rightarrow\R^m$ (which is of course also not a new result, in view of Theorem \ref{rago} and Theorem \ref{clos}). Use the well-known fact that every globally nonnegative quadratic polynomial is a sum of squares of affine linear polynomials, and apply Proposition \ref{trans} with $S=\R^n$ and $V=\R[\ul X]_2.$ Again recall that $P=\sum\R[\ul X]_1^2\subseteq V$ is the projection of a spectrahedron.
\end{Ex}
In the following result, case (i) for a full rational curve is proven in Henrion \cite{He}, Theorem 1. In the version it is stated here it has also been the topic of a talk of Parrilo at a workshop in Banff in 2006, but there seems to be no suitable reference. Case (ii) relies on results of Scheiderer, as also explained in \cite{genone}. 
\begin{Cor}\label{curve} Let $S\subseteq\R^n$ be either \begin{itemize} \item[(i)] a semi-algebraic subset of a rational curve, or \item[(ii)] a smooth curve of genus $1$ with at least one non-real point at infinity. \end{itemize} Then for any rational map $$f=\left(\frac{f_1}{g},\ldots, \frac{f_m}{g}\right)\colon\R^n\rightarrow\R^m$$ that is defined everywhere on $S,$  we find that  $$\overline{\conv(f(S))}$$ is the projection of a spectrahedron.
\end{Cor}
\begin{proof}
First check that we can reduce to the case $g=1,$ i.e. the case that $f$ is a polynomial map.  
Indeed for a general rational map $f$ we can take without loss of generality a denominator $g$ that is positive on $S$, and we can also prove the claim for the following map instead: $$F\colon S\rightarrow\R^{m+1};\ x \mapsto \left(\frac{f_1(x)}{g(x)},\ldots, \frac{f_m(x)}{g(x)},1\right).$$ Then define $$G\colon S\rightarrow \R^{m+1};\ x \mapsto g(x)\cdot F(x).$$ This map is polynomial and thus assume we already know that $\overline{\conv(G(S))}$ is the projection of a spectrahedron. By \cite{NeSi}, Proposition 2.1, the conic hull of the projection of a spectrahedron is again such a projection. So together with Corollary \ref{clos} we get that $\overline{\cc(G(S))}$ is the projection of a spectrahedron. But now one checks $$\overline{\cc(G(S))}\cap\left(\R^m\times\{1\}\right)=\overline{\conv(F(S))},$$ which finishes the reduction step.

Now we ensure the existence of some finitely generated quadratic module $\QM(\ul p)$ in $\R[\ul X]$ that contains all polynomials nonnegative on $S$, with a degree bound on the sums of squares depending only on the degree of the respective polynomial. Then we can apply Proposition \ref{trans} with an arbitrary finite dimensional space $V$ and a suitable $\QM(\ul p)_d$.

Now for (i) it is clearly enough to consider the case of a semialgebraic subset of a straight line, which is covered by Kuhlmann, Marshall and Schwartz \cite{MR2174483} Theorem 4.1 (see also the paper by Scheiderer \cite{MR1829790}).

For (ii) it is enough to ensure the existence of such bounded degree representations in a quadratic module modulo the vanishing ideal $(g_1,\ldots, g_k)$ of the respective curve. Indeed, if some polynomial $p$ has a representation $$p=\sum_i\si_ip_i + \sum_j h_jg_j$$ with sums of squares $\si_i$ and arbitrary polynomials $h_j$, and the degree of the $\si_i$ is bounded by $2d$, then one can find a similar representation with polynomials $\widetilde{h}_j$ of  a degree bounded by some number not depending on the specific choice of $p$. This follows from the fact that an ideal intersected with a finite dimensional subspace of $\R[\ul X]$ is a finite dimensional subspace, and one can choose a finite basis. So if $p-\sum_i\si_ip_i$ belongs to that space, it is an $\R$-linear combination of these finitely many basis elements. This yields  a representation with polynomials $\widetilde{h}_i$ as desired. Then the quadratic module $\QM(\ul p, \pm \ul g)$ has the  property that we claimed in the beginning, using the standard equality $$h_j=\left(\frac{h_j+1}{2}\right)^2-\left(\frac{h_j-1}{2}\right)^2$$ for any polynomial $h_j$.

Now for smooth genus one curves with a non-real point at infinity the pure existence of sums of squares representations is Scheiderer \cite{MR1675230}, Theorem 4.10 (a), whereas the degree bounds are explained in Scheiderer \cite{genone}.
\end{proof}

Note that this result applies in cases where the curvature results from Helton and Nie \cite{HeltonNieSDPrepr} \cite{HeltonNieNecSuffSDP} and Lasserre's direct approach from \cite{LasserreConvSets} do not apply:

\begin{Ex} The basic closed semi-algebraic set $S=\{(x,y)\in\R^2\mid 0\leq y\leq 1, -1\leq x, y^2-x^3\geq 0\}$ is bounded by segments of rational curves. So it is the projection of a spectrahedron. The results from Helton and Nie do not apply since $Y^2-X^3$ is neither strictly quasi-concave on $S$, nor sos-concave. Also it is singular at the origin. The standard Lasserre method does not apply since $S$ has a nonexposed face, see for example Theorem \ref{exp} below. One could also replace the part of the set on the left hand side of the $y$-axis by a half disk. The resulting set is then even not basic closed, and still the Theorem applies.
\end{Ex}

\begin{Ex} Let $S\subseteq\R^2$ be defined by the inequality $y^2\leq 1-x^4$. The boundary is a smooth genus one curve with a non-real point at infinity. Thus $S$ is the projection of a spectrahedron. Applying the polynomial map $(x,y)\mapsto (x^2,xy^2)$ sends this curve to the 
boundary of the convex set $y^2 \leq x - 2x^3 +x^5$ which has a singularity at the point $(1,0)$ as seen in Figure \ref{fig:rational_map}. 
Still Corollary \ref{curve} guarantees that this set is the projection of a spectrahedron.
\end{Ex}

\begin{figure}[ht]
\begin{center}
\hfill 
\includegraphics[scale=0.25]{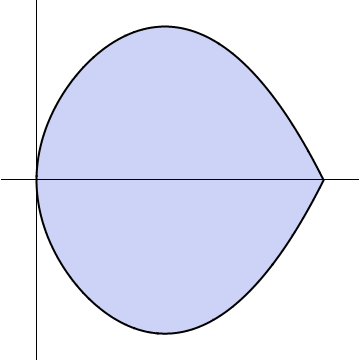}\hfill\
\caption{\small Convex Hull of the image of the curve $1-x^4-y^2=0$ by the map $(x,y)\mapsto (x^2,xy^2)$. } \label{fig:rational_map}
\end{center}
\end{figure}

We state some more corollaries of Proposition \ref{trans}. The following result is  Henrion \cite{He}, Theorem 1:

\begin{Cor}\label{quad} Let either $f\colon\R^3\rightarrow\R^m$ be homogeneous of degree $4$ or $f\colon\R^2\rightarrow\R^m$ of degree $4$ (but not necessarily homogeneous). Then the closure of the convex hull of the image of $f$ is the projection of a spectrahedron .
\end{Cor}
\begin{proof}  We can apply Proposition \ref{trans}, using Hilbert's result that every globally nonnegative homogeneous degree $4$ polynomial in three variables and every globally nonnegative degree $4$ polynomial in two variables is a sum of squares.
\end{proof}

We get another result that has to our knowledge not been observed before:

\begin{Cor} Let $f\colon\R^4\rightarrow\R^m$ be homogeneous quadratic. Let $C\subseteq\R^4$ be any polyhedral cone. Then $\overline{\conv(f(C))}$ is the projection of a spectrahedron.
\end{Cor}
\begin{proof} Every polyhedral cone in $\R^4$ is a finite union of cones that can be transformed by a linear automorphism to the first orthant in some $\R^k$ with $k\leq 4$. This  follows from Caratheodory's Theorem for cones. If $C=C_1\cup\cdots\cup C_m$ then $$\overline{\conv(f(C))}=\overline{\conv(\overline{f(C_1)}\cup\cdots\cup \overline{f(C_m)})}.$$ So by the convex hull result from Helton and Nie \cite{HeltonNieNecSuffSDP} (see also \cite{NeSi}) and Theorem \ref{clos} it is enough to prove the Theorem for the first orthant $C$ in $\R^4$.

Every quadratic form in  $4$ variables that is nonnegative on the first orthant belongs to the quadratic module generated by the pairwise products of the variables $X_iX_j$. This is just a slight reformulation of the main result from Diananda \cite{MR0137686}. But then a degree bound condition on the sums of squares is fulfilled for any such representation, since no degree cancellation can occur when adding polynomials that are nonnegative on the first orthant.  So in fact each such nonnegative quadratic form is a positive combination of the $X_iX_j$ plus a sums of squares of linear forms. Now apply Proposition \ref{trans} with $\ul{p}=\{X_iX_j\mid 1\leq i,j\leq n\}$ and $V$ the space spanned by the quadratic forms and $1$.
\end{proof}

\section{Obstructions to the relaxation methods}\label{ex}
In this section we examine the assumption from Lasserre's Theorem, as stated in Theorem \ref{lasmain} above. That means, we want to know whether there exists some $d$ such that the truncated quadratic module $\QM(\ul p)_d$ contains all nonnegative linear polynomials. Note here that this condition is absolutely not  necessary for $\overline{\conv(S)}$ to be the projection of a spectrahedron. This is for example shown by Example 3.7 in \cite{NePlSch} (that we will discuss in more detail below). But in view of Theorem \ref{lasmain}, the condition is necessary and sufficient for the Lasserre approach to work. This brings up the question when this so called \textit{bounded degree representation property} for affine polynomials is fulfilled. 

A necessary condition is given by the following result:

\begin{Prop}\label{aux}
Let $p_1,...,p_r \in \R[\ul X]$, $S=\sS(\ul p)$ and $L$ be a line in $\R^n$ such that $S \cap L$ has non-empty interior relative to $L$.
Let $a\in S$ be a point that belongs to the relative  boundary of $\overline{\conv(S)} \cap L$ in $L$.  Assume that for all $p_i$ with 
$p_i(a)=0$ the vector $\bigtriangledown  p_i(a)$ is orthogonal to $L$. Then, for all $d$, the Lasserre relaxation $\sL(\ul p)_d$ strictly contains $\overline{\conv(S)}$.
\end{Prop}
\begin{proof}
By applying a linear transformation we may assume $L$ to be the $X_1$-axis, $a$ to be the origin and $\overline{\conv(S)} \cap L$ to be on the positive
half axis. Let $p_1',\ldots,p_r'\in \R[X_1]$ be the polynomials obtained from $p_1,\ldots,p_r$ by setting the last $n-1$ variables to zero. 
We have $\sL(\ul p')_d \subseteq \sL(\ul p)_d \cap \R$ for any $d$. This inclusion follows from the fact that each polynomial $f\in\QM(\ul p)_d$ 
ends up in $\QM(\ul p')_d$ when setting the last $n-1$ variables to zero.

Let $S'=\sS(\ul p')$, so $S'= S\cap \R$. If $\conv(S')$ is some closed interval $[0,c]$ then let $p'_{r+1}=c-X_1$, 
otherwise (i.e. if $\conv(S')=[0, \infty)$) let $p'_{r+1}=1$ (just to
keep the notation uniform). Then $\sS(\ul p')=\sS(\ul p',p'_{r+1})$. Note that $\sL(\ul p',p'_{r+1})_d \subseteq \sL(\ul p')_d$, 
and since $S'$ has an interior point $\sL(\ul p',p'_{r+1})_d = \conv(S')$ holds if and only if every nonnegative affine 
linear polynomial from $\R[X_1]$  belongs to $\QM(\ul p',p'_{r+1})_d$, by Theorem \ref{lasmain}.
Consider the polynomial $X_1$, that is nonnegative on $S'$, and suppose there exists a representation
$$X_1 = \si + \sum_{i \in I} \sigma_i p'_i + \sum_{j \in J} \sigma_j p'_j,$$
where $i \in I$ if $p_i(0)>0$ and $i \in J$ otherwise. For $i\in I$,  $p'_i$ has a positive constant term, so
 $\sigma_i$ cannot have a constant term, and its homogeneous part of minimal degree must be at least quadratic. The same is true for $\si$.
So none of the elements $\si$ and $\si_ip'_i$ where $i\in I$ contains the monomial $X_1$. But by hypothesis, $\bigtriangledown  p_j(0)$ is orthogonal
to the $x_1$-axis for $j\in J$, which implies that the terms of $p'_j$ have all degree at least $2$. This is a contradiction. So $\sL(\ul p',p'_{r+1})_d$ is not
$\conv(S')$. Since $p'_{r+1} \in \QM(\ul p',p'_{r+1})_d$ this implies the existence of some negative $b$ with 
$b \in \sL(\ul p',p'_{r+1})_d \subseteq \sL(\ul p')_d \subseteq L \cap \sL(\ul p)_d$. But since $b \not \in \overline{\conv(S)}$ by hypothesis,
this implies $\sL(\ul p)_d \not = \overline{\conv(S)}$.
\end{proof}

This gives an alternative and more elementary proof to Theorem 3.5 in \cite{NePlSch}:

\begin{Thm}\label{exp}
 Let $p_1,...,p_r \in \R[\ul X]$ be such that $S=\sS(\ul p)$ is convex and has non-empty interior. If $S$ has a non-exposed face, then for all $d$, the Lasserre relaxation $\sL(\ul p)_d$ strictly contains $S$.
\end{Thm}
\begin{proof}
Let $F\subseteq S$ be a non-exposed face. Then there exists some face $F_1$ of $S$, such that $F \subsetneq F_1$ and for all supporting hyperplanes 
$H$ containing $F$, $F_1 \subseteq H$. Let $a$ be a point in the relative interior of $F$ and $L$ a line passing through $a$ and some point 
in the relative interior of $F_1$. By convexity and closedness of $S$ we have that $a$ belongs to the relative boundary of $\overline{\conv(S)}\cap L$, and we just have to verify the gradient condition at $a$.

Suppose $p_j(a)=0$, and consider $v:=\bigtriangledown  p_j(a)$.
For any $x \in \R^n$ the product  $v \cdot (x-a) $ equals the derivative of $p_j$ at $a$ in direction of $(x-a)$, so by convexity of $S$ we get 
$v\cdot (x-a)\geq 0$ whenever $x\in S.$ Hence the linear polynomial $\ell:=v_1(X_1-a_1) + \cdots + v_n(X_n-a_n)$ is 
nonnegative on  $S$. Since $\ell$ vanishes at $a,$ which lies in the relative interior of $F$, it vanishes on the whole of $F$ and thus also on $F_1$. 
Then, since it vanishes in two points of $L$, it must vanish on the entire line, which implies that $v$ is orthogonal to $L$, and Proposition \ref{aux} gives us the result.
\end{proof}

The lemma also shows us the following result:

\begin{Thm} \label{prop:singularity}
Let $p_1,\ldots, p_r\in\R[\ul X]$ and let $S:=\sS(\ul p)\subseteq\R^n$ have non-empty interior. Let $a$ be point in the boundary of $\overline{\conv(S)}$
that is also in $S$, and suppose that all active constrains at $a$ are singular. Then for all $d$, the Lasserre relaxation $\sL(\ul p)_d$ strictly contains $\overline{\conv(S)}$. 
\end{Thm}
\begin{proof}
Just consider a line $L$ passing through $a$ and through the interior of $S$ and apply Proposition \ref{aux}.
\end{proof}

\begin{Ex} Consider the semi-algebraic set $S=\{ (x,y) \in \R^2: p(x,y):=-x^4+x^3-y^2 \geq 0\}$. The convex hull of $S$ is intersected by the 
$x$-axis in the segment $[0,1]$, which has non-empty interior. Furthermore $p$ has a singularity at the origin, hence we are in the conditions of
 Theorem \ref{prop:singularity} and the Lasserre hierarchy does not converge in finitely many steps, although it does approximate the set $\conv(S)$
as shown in Figure \ref{cusp}.
 
\begin{figure}[ht]
\begin{center}
\hfill 
\includegraphics[scale=0.3]{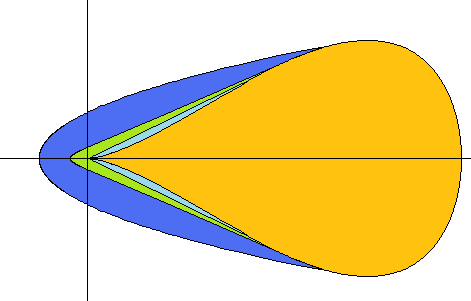}\hfill\
\caption{\small From the smallest to the largest: the sets $S$, $\conv(S)$, $\sL(\ul p)_2$ and $\sL(\ul p)_1$. } \label{cusp}
\end{center}
\end{figure}
\end{Ex}
 
The same general idea we used for the Lasserre relaxations can also be applied to the theta body construction. To do that, however, we need some
auxiliary definitions. Let $I$ be any ideal, and $p$ a point in $\V_{\R}(I)$. The {\bf tangent space} $T_{p}(I)$ is the affine space through $p$
that is orthogonal to the space spanned by the gradients of all polynomials in $I(\V_{\R}(I))$, the vanishing ideal of $\V_\R(I)$. We say that a point $p \in \V_{\R}(I)$ 
on the boundary of
$\conv(\V_{\R}(I))$ is {\bf convex-non-singular} if $T_{p}(I)$ is tangent to $\conv(\V_{\R}(I))$ i.e., if it does not intersect its relative interior;
otherwise we say that $p$ is {\bf convex-singular}.

\begin{Thm} \label{prop:theta_singularity}
Let $I$ be any ideal such that $\V_{\R}(I)$ has a convex-singular point, then for all $d$ $\TH_d(I)$ strictly
contains $\overline{\conv(\V_{\R}(I))}$.
\end{Thm}
\begin{proof}
Let $J$ be the vanishing ideal of $\V_{\R}(I)$. Since $I$ is contained in $J$, $\TH_d(J) \subseteq \TH(I)_d$, so it is enough to show that 
$\TH_d(J) \not = \overline{\conv(\V_{\R}(I))}$. Suppose we have equality. 
Since $J$ is real radical, Theorem \ref{thm:thetamain} tell us that any linear polynomial that is
nonnegative in $\V_{\R}(I)$ must be in $\Sigma(d,J)$. Let $p$ be the convex-singular point of $\V_{\R}(I)$. Since $p$ is in the boundary
of $\conv(\V_{\R}(I))$  there exists a linear polynomial $\ell$ that is zero in $p$ and positive on the relative interior of  $\conv(\V_{\R}(I))$. Therefore $\ell= \sigma + g$
where $\sigma$ is a sum of squares and $g \in J$. Let $q$ be a point in $T_{p}(I)$ that is in the relative interior of $\conv(\V_{\R}(I))$.
We have 
$$(q-p) \cdot \nabla \ell(p)= (q-p) \cdot\nabla \sigma(p) + (q-p) \cdot \nabla g(p).$$
But since $\sigma$ is a sum of squares vanishing at $p$, it must have a double zero there so its gradient also vanishes there, and since $q$ 
belongs to $T_p(I)$ then for all $g \in J$, $(q-p)$ is orthogonal to their gradients at $p$, so we have that the derivative of $\ell$ in the
direction of $(q-p)$ is zero. Since $\ell$ is linear, this implies that it vanishes at $q$, which is a contradiction.
\end{proof}

\begin{Rem}
Note that if $J= I(\V_{\R}(I))$ is generated by a single polynomial (so $\V_{\R}(I)$ is a hypersurface), then any singular point $p$ from $\V_\R(I)$ that belongs to the boundary of $\conv(\V_{\R}(I))$ is convex-singular. This is clear since the tangent space at $p$ is the whole of $\R^n$ in that case.
\end{Rem}

\begin{Ex}(i)
An example for the above remark is the (compact) Zitrus surface defined by $x^2+z^2 +(y^2-1)^3=0$ in $\R^3$. It has a singularity at $(0,1,0),$ which belongs to the boundary of the convex hull, and thus each theta body relaxation strictly contains the convex hull of the surface. The boundary equations for the convex hull of this surface have been examined in detail by Sturmfels and Ranestad in \cite{sr}, Section 4.2.

(ii) Consider the variety $V_\R(I)$ in $\R^3$ defined by the ideal $$I=\langle x^2+y^2+z^2-4, (x-1)^2+y^2-1\rangle.$$ It has a singularity at the point $p=(2,0,0)$, which belongs to the boundary of the convex hull of $\V_\R(I)$. This singularity is however not convex-singular, as one easily checks. And indeed already the first theta body relaxation equals $\conv(\V_\R(I)).$ To see this first note that $I$ can also be defined by $p_1=(x-1)^2+y^2-1$ and $p_2=2x +z^2-4$. Write $I_1=\langle p_1\rangle$ and $I_2=\langle p_2\rangle$. Then note  $$\conv(\V_\R(I))=\conv(\V_\R(I_1))\cap \conv(\V_\R(I_2)).$$ Since $\TH_d(I)\subseteq \TH_d(I_1)\cap \TH_d(I_2)$ holds obviously, it is enough to show that the theta body relaxations for $I_1$ and $I_2$ are exact in the first step. But this follows for example from Lemma 5.5. in \cite{GPT}, since $p_1$ and $p_2$ are convex quadrics. The example shows that the notion of a convex-singular point is crucial in Theorem \ref{prop:theta_singularity}.
\end{Ex}

We go back to Theorem \ref{exp}. It says that a convex basic closed set $S$ can only equal some relaxation $\sL(\ul p)_d$ if all of its faces are exposed. In \cite{NePlSch} the question is raised whether this can be generalized:

\begin{Question}{\cite[Remark 3.8]{NePlSch}}\label{q1}
\begin{itemize}
\item[(i)] Is Theorem \ref{exp} still true with $S$ replaced by $\overline{\conv(S)},$ if $S$ is non-convex?
\item[(ii)] More generally, are all faces of $\sL(\ul p)_d$ exposed for all $d$ and $\ul p$?
\end{itemize}
\end{Question}
One can also ask if Theorem \ref{exp} can be generalized to the theta body relaxations:
\begin{Question}\label{q2}
Let $I\subseteq \R[\ul X]$ be an ideal such that $\TH(I)_d=\overline{\conv(\V_{\R}(I))}$. Are all faces of $\TH(I)_d$ exposed faces in this case?
\end{Question}

The answer to all these questions is negative, as we will show.

\begin{Prop} \label{counterexample}
Let $p_1=Y, p_2=1-Y,p_3=Y-X^3,p_4=1+X$ define the set $S=\sS(\ul p)\subseteq\R^2$.  Then $\sL(\ul p)_1$ is the convex hull of $S \cup \{(1/3,0)\}$.
\end{Prop}
\begin{proof}
Let $C= \conv(S \cup \{(1/3,0)\})$. Then $C$ is cut out by the infinitely many affine linear inequalities 
$$\{Y \geq 0, 1-Y \geq 0, 1 + X \geq 0, Y-3a^2X +2a^3 \geq 0\mid a \in[1/2,1]\},$$
since the polynomial $\ell_a:=Y-3a^2X +2a^3$ defines the half-plane containing $S$
and tangent to the curve $Y=X^3$ at the point $(a,a^3)$. To prove $\sL(\ul p)_1 \subseteq C$ it is thus enough to
show that the polynomials $\ell_a$ belong to  $\QM(\ul p)_1$ for all  $a \geq 1/2$. To see this, note that
$$\ell_a=(\sqrt{2a-1}(X-a))^2 + (Y-X^3) + (X-a)^2(X+1).$$

To prove the inclusion $C \subseteq \sL(\ul p)_1$, using the fact that $\sL(\ul p)_1$ is convex and contains $S$, it is enough to show
that $(1/3,0) \in \sL(\ul p)_1$. Since translations commute with taking Lasserre relaxations, we will instead consider the set of polynomials
$p_1'=Y,p_2'=1-Y,p_3'=X+4/3,p_4'=Y-X^3-X^2-X/3-1/27$ obtained from the $p_i$ by replacing $X$ by $X+1/3$, and prove that $(0,0) \in \sL(\ul p')_1$. Suppose
that is not the case. Then there must exist $\epsilon,\mu>0$ such that $\ell=Y-\mu X - \epsilon$ 
belongs to $\QM(\ul p')_1$. This means
$$\ell=\si_0+\si_1Y+\si_2(1-Y)+\si_3(X+4/3)+c(Y-X^3-X^2-X/3-1/27),$$
where $c$ is simply a nonnegative constant, since $\deg(p'_4)=3$.
Note that $\si_0$ has at most degree $2$, as do $\si_1,\si_2$ and $\si_3$. 

Let $\si_3=a_1 X^2 + a_2 X + a_3 + a_4 Y^2 +a_5 XY + a_6Y$. In order to cancel the $X^3$ term
of the entire expression, we must have $a_1=c$. The coefficient for $X^2$ will then be $a-c+4/3c +a_2$, where $a$ is a nonnegative number which
is the sum of the coefficients of $X^2$ in $\si_0$ and $\si_2$. This implies $a_2 \leq -c/3$, which by using the fact that $\si_3$ is a sum of squares,
implies $a_3 \geq c/36$ (just consider a Hankel matrix for this sum of squares and analyze the submatrix indexed by $1$ and $x$).

Now checking the constant coefficient, we will have it to be $b-c/27 + 4a_3/3,$ where $b$ is the nonnegative constant term of $\si_0+\si_2$. Since this
must be $-\epsilon$, we have $-c/27 + 4a_3/3 < 0$ which since $a_3\geq c/36$ is impossible. Hence $\ell\notin \QM(\ul p')_1$, and $(0,0)$ is in $\sL(\ul p)_1$ as intended.
\end{proof}

\begin{figure}[ht]
\begin{center}
\hfill \includegraphics[scale=0.3]{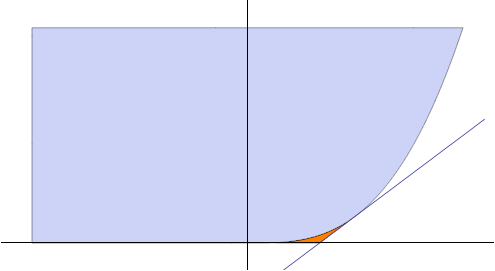} \hfill
\includegraphics[scale=0.22]{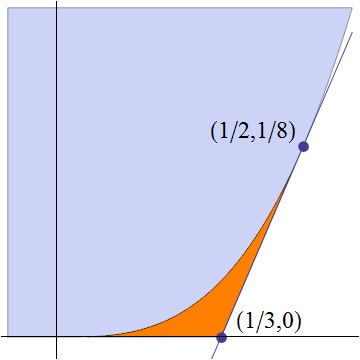}\hfill\
\caption{\small Comparison between $S$ and $\sL(\ul p)_1$, where  $p_1=Y,p_2=1-Y,p_3=1+X,p_4=Y-X^3$. Full region on the left, close up on the right.} \label{fig1}
\end{center}
\end{figure}

\begin{Cor}
For $p_1=Y,p_2=1-Y,p_3=1+X,p_4=Y-X^3$, $\sL(\ul p)_1$ has a non-exposed face at $(1/2,1/8)$ .
\end{Cor}
\begin{proof}
Immediate, from Figure \ref{fig1}.
\end{proof}

This shows that general Lasserre relaxations might have non-exposed faces, giving a negative answer to  Question \ref{q1} (ii).
In fact, this can happen even for very \textquotedblleft well-behaved\textquotedblright \ semialgebraic sets.
If in Proposition \ref{counterexample} we change the defining polynomials $\ul p$ to $\ul p'$ by replacing $Y$ with $Y-1/10$, we get a semialgebraic set  
that has only exposed faces (it can even be shown that $\sS(\ul p')=\sL(\ul p')_2$). However, our proof still works in this case, showing 
that $\sL(\ul p')_1=\sL(\ul p)_1 \cap \{(x,y)\mid y>1/10\}$ has a non-exposed face.

In the next proposition we show that when $\sS(\ul p)$ is not
convex, even if one of its Lasserre relaxations is tight (meaning $\sL(\ul p)_d=\overline{\conv(S)}$ for some $d$), $\sL(\ul p)_d$ might still have non-exposed faces.

\begin{Prop}\label{example2}
For $p:=-X^4-Y^4-2X^2Y^2+4X^2\in\R[X,Y]$ we find $$\sL(p)_2=\conv(\sS(p)).$$
\end{Prop}
\begin{proof}
The set $S=\sS(p)$ is the union of two disks of radius $1$ with centers $(-1,0)$ and $(1,0)$. By symmetry, it is enough to show that any linear polynomial
tangent to the left circle and non-negative on both disks belongs to $\QM(p)_2$. The points on the left circle that
are on the boundary of $\conv(S)$ are of the form $z_{\theta}:=(\cos(\theta)-1,\sin(\theta))$, for some $\theta \in[\pi/2,3\pi/2]$, 
and an affine linear polynomial $\ell_{\theta}$ defining the tangent to $z_{\theta}$ such that $\ell_{\theta} \geq 0$ on $S$ is given by 
$\ell_{\theta}=1-\cos(\theta)-\cos(\theta)X -\sin(\theta)Y$. Since $\cos(\theta)\leq 0$
 it is enough to check that the equality
\begin{equation}\label{eq2circ}
\begin{array}{ll}
 (8-8\cos(\theta))\ell_{\theta} &  = p +(X^2+Y^2-2+2\cos(\theta))^2 + \\
&   \hspace{.5cm}+\left(2\sqrt{1-\cos(\theta)}(Y-\sin(\theta))\right)^2+ \\
& \hspace{.5cm} + \left(2\sqrt{-\cos(\theta)}(X-\cos(\theta)+1)\right)^2
\end{array}
\end{equation}
holds, thus proving the result.
\end{proof}

\begin{figure}[ht]
\begin{center}
\includegraphics[scale=0.3]{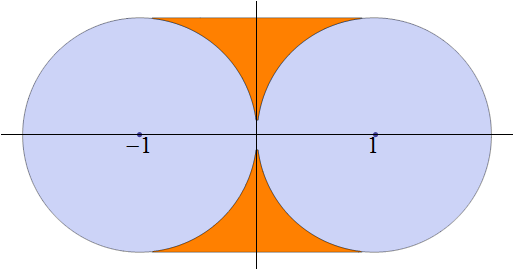}
\caption{\small Comparison between $S$ and $\sL(p)_2=\conv(S)$, where \newline $p=-X^4-Y^4-2X^2Y^2+4X^2$.} \label{fig2}
\end{center}
\end{figure}

\begin{Cor}
For $p=-X^4-Y^4-2X^2Y^2+4X^2$, $\sL(p)_2=\conv(\sS(p))$ has a non-exposed face.
\end{Cor}
\begin{proof}
Just note that the four points $(\pm 1, \pm 1)$ are all non-exposed faces of $\sL(p)_2=\conv(\sS(p))$ as it can be seen in Figure \ref{fig2}.
\end{proof}

Note that the proof of Proposition \ref{example2} not only completes the answer to Question \ref{q1} (i), but also answers Question \ref{q2}. 
Our representation \eqref{eq2circ} shows that if we consider the ideal $I=\left<p\right>$, then $\TH(I)_2=\conv(V_{\R}(I))$ has 
non-exposed faces.

\section{A Positivstellensatz for projections of spectrahedra}\label{third}

In this section we describe a quadratic module that is assigned to the projection of a spectrahedron. This quadratic module will in general not be finitely generated, but still its elements can be described almost constructively. The module will turn out to be archimedean whenever the set is bounded, and it will thus provide us with a Positivstellensatz for projections of spectrahedra. This is in particular interesting since such projections are usually not basic closed semialgebraic. So none from the large amount of  present Positivstellens\"atze applies to this setup.

 Another interesting feature of this quadratic module is that it establishes a counterpart to Lasserre's theorem above. Recall that the existence of a \textit{finitely generated} quadratic module containing all nonnegative linear polynomials with a degree bound on the sums of squares is only sufficient, but not necessary for a set to be the projection of a spectrahedon. The module that we will construct, however, contains all nonnegative linear polynomials in a certain truncated part. So if we broaden the class of quadratic modules from finitely generated ones to a certain bigger class, then the bounded degree representation property from Lasserre's theorem becomes equivalent to representability  of a set as the projection of a spectrahedron.

So let $\sA(\ul X,\ul Y)=A+X_1B_1+\cdots+X_nB_n+Y_1C_1+\cdots+Y_mC_m$ be a strictly feasible $k$-dimensional linear matrix polynomial. Let $\widetilde{S}\subseteq\R^{n+m}$ be the spectrahedron defined by $\sA$, and $S=\pr(\widetilde{S})\subseteq\R^n$ its projection. We will write $\sA'(\ul X)$ for $\sA(\ul X,0)$. 

Recall that any linear polynomial $\ell\in\R[\ul X]_1$ that is nonnegative on $S$ is of the form $$\ell = U\circ \sA'(\ul X) + r ,$$ with some $r\geq 0$ and  a positive semidefinite $k\times k$-matrix $U$ that fulfills $U\circ B_i=0$ for all $i=1,\ldots, m$. This is precisely the statement of Proposition \ref{nem}. By Cholesky decomposition of $U$ this is the same as saying $$\ell = \sum_j v_j^t\sA'(\ul X)v_j +r $$ for finitely many vectors $v_j\in\R^k$ fulfilling $\sum_j v_j^tB_i v_j=0$ for all $i=1,\ldots, m$. 

If we now want to construct a quadratic module containing all the nonnegative linear polynomials on $S$, we can use polynomial vectors $q_j$  instead of real vectors $v_j$ only. Formally, define \begin{align*} \QM(\sA):=\left\{ \sum_j q_j^t \sA'(\ul X) q_j+ \si\mid \right. q_j\in\R[\ul X]^k, & \sum_j q_j^tB_i q_j=0 \mbox{ for } i=1,\ldots m, \\ &\left. \si\in\sum\R[\ul X]^2\right\}.\end{align*}  
 Clearly $\QM(\sA)$ is a quadratic module. The following main result now follows easily. In the case of a bounded set $S$ it provides the announced Positivstellensatz.
 
 \begin{Thm} \label{pos}$\QM(\sA)$ contains only polynomials that are nonnegative on $S$, and the set of points in $\R^n$ where all elements from $\QM(\sA)$ are nonnegative is precisely $\overline{S}$. If $S$ is bounded then $\QM(\sA)$ is archimedean, and thus contains all polynomials $p$ with $p> 0 $ on $\overline{S}$.
 \end{Thm}
 \begin{proof} The first statements follows immediately from the fact that each element from $\QM(\sA)$ is in particular of the form $$\sum_j q_j^t\sA(\ul X,\ul Y)q_j + \si,$$ and from the definition of $S$. The second statement is then clear from the fact that all nonnegative linear polynomials are contained in $\QM(\sA)$. In the case of a bounded set $S$ we  have $N\pm X_i\in\QM(\sA)$ for all $i$ and some sufficiently large number $N$. As for example explained in Marshall \cite{MR2383959}, Corollary 5.2.4, $\QM(\sA)$ is archimedean. Then Jacobi's Representation Theorem \cite[Theorem 4]{MR1838311}  implies the statement about strictly positive polynomials.
 \end{proof}

Note that in case of a spectrahedron,  Helton,  Klep, and McCullough \cite{hkm} have also proven $\QM(\sA)$ to be archimedean, using results about completely positive maps. They use this to obtain a Positivstellensatz for matrix polynomials, see their Theorem 1.3. 

Note also that in our result we can not expect $\QM(\sA)$ to be a finitely generated quadratic module in general. This would imply that $\overline{S}$ is basic closed semi-algebraic, i.e. defined by finitely many simultaneous polynomial inequalities. This is clearly not true for all projections of spectrahedra.

\begin{Ex} Consider the example  from Proposition \ref{example2}, the convex hull of two disks in the plane. In contrast to the above example, we denote by $S$ the full convex hull. Note that $S$ is an example of a closed semi-algebraic set that is not basic closed. Since $S$ is the union of disks shifted along the $x$-axis, one immediately checks that it has the following representation: $$S=\{(x,y)\in\R^2\mid \exists z\in [-1,1]\ (x-z)^2+y^2\leq 1\}.$$
 The defining condition of $S$ can now be stated as positive semidefiniteness of the following linear matrix polynomial:

 $$ \left(\begin{matrix}1 & 0 & 0 & 0 \\0 & 1 & 0 & 0 \\0 & 0 & 1 & 0 \\0 & 0 & 0 & 1\end{matrix}\right)  + X \left(\begin{matrix}0 & 1 & 0 & 0 \\1 & 0 & 0 & 0 \\0 & 0 & 0 & 0 \\0 & 0 & 0 & 0\end{matrix}\right) + Y \left(\begin{matrix}1 & 0 & 0 & 0 \\0 & -1 & 0 & 0 \\0 & 0 & 0 & 0 \\0 & 0 & 0 & 0\end{matrix}\right) +Z \left(\begin{matrix}0 & -1 & 0 & 0 \\-1 & 0 & 0 & 0 \\0 & 0 & 1 & 0 \\0 & 0 & 0 & -1\end{matrix}\right) $$

So by Theorem \ref{pos}, every polynomial that is strictly positive on $S$ is a sum of squares plus a polynomial of the following form:   $$\sum_j q_1^{(j)^2} +q_2^{(j)^2}+q_3^{(j)^2}+q_4^{(j)^2}+ X(q_3^{(j)^2}-q_4^{(j)^2}) + Y(q_1^{(j)^2}-q_2^{(j)^2}),$$ where $q_i^{(j)}\in\R[X,Y]$ with $\sum_j 2q_1^{(j)}q_2^{(j)}-q_3^{(j)^2}+q_4^{(j)^2}=0$.

\end{Ex}

We now turn to the announced counterpart of Lasserre's Theorem. 
First consider the following truncated part of $\QM(\sA)$:  \begin{align*} \QM(\sA)_d:=\left\{ \sum_j q_j^t \sA'(\ul X) q_j+ \si\mid \right. q_j\in(\R[\ul X]_d)^k, & \sum_j q_j^tB_i q_j=0 \mbox{ for } i=1,\ldots m, \\ &\left. \si\in\sum\R[\ul X]_d^2\right\}.\end{align*}  

\begin{Lemma}\label{trunc}
Each $\QM(\sA)_d$ lives in a finite dimensional subspace of $\R[\ul X]$ and is the projection of a spectrahedron.
\end{Lemma}
\begin{proof} It is clear that $\QM(\sA)_d$ lives in a finite dimensional subspace. Now for finitely many $k$-tuples $q_1,\ldots, q_r$ of polynomials consider the $k\times k$-matrix polynomial $$M= q_1q_1^t+\cdots+ q_rq_r^t.$$ The condition $\sum_j q_j^tB_iq_j=0$ translates to $B_i\circ M=0,$ and $$\sum_j q_j^t\sA'(\ul X)q_j = \sA'(\ul X)\circ M.$$ If the degree of all components of the $q_i$ is bounded by $d$, then the degree of each entry of $M$ is bounded by $2d$, and thus $M$ can be written as a sum $$M=p_1p_1^t + \cdots +p_Np_N^t,$$ with some $N$ depending on $d$ and $k$, but not on $r$. This follows from Caratheodory's Theorem. Now consider the quadratic mapping  \begin{align*}\Psi\colon ((\R[\ul X]_d)^k )^N &\rightarrow \M_k(\R[\ul X]_{2d}) \\ (p_1,\ldots,p_N)&\mapsto \sum_j p_jp_j^t.\end{align*} Its image is a convex cone, and the projection of a spectrahedron by Theorem \ref{rago}. So intersecting with the linear subspace defined by $M\circ B_i=0$ for $i=1\ldots,m$ and applying the linear map $M\mapsto \sA'(\ul X)\circ M$ still gives the projection of a spectrahedron. After taking the convex hull with $\sum\R[\ul X]_d^2$ we obtain $\QM(\sA)_d$, still the projection of a spectrahedron.
\end{proof}

\begin{Thm}\label{genla} Let $S\subseteq\R^n$ be a set such that $\overline{\conv(S)}$ has nonempty interior. Then the following are equivalent:
\begin{itemize}
\item[(i)] $\overline{\conv(S)}$ is the projection of a spectrahedron.
\item[(ii)] There is a quadratic module $Q\subseteq\R[\ul X]$ with the properties
\begin{itemize}
\item[$\bullet$] $Q$ contains only polynomials nonnegative on $S$
\item [$\bullet$] $Q=\cup_{d\in\N} Q_d$, where $Q_d\subseteq Q_{d+1}$ and each $Q_d$ lives in a finite dimensional subspace of $\R[\ul X]$ and is the projection of a spectrahedron.
\item[$\bullet$] There is some $d^*$ such that $Q_{d^*}$ contains every affine linear polynomial that is nonnegative on $S$.
\end{itemize}
\end{itemize}
\end{Thm}

\begin{proof} For "(ii)$\Rightarrow$(i)" consider the set $Q_{d^*}\cap\R[\ul X]_1$ in $\R[\ul X]_1$. It is the projection of a spectrahedron, and so \begin{align*}\overline{\conv(S)}&=\{x\in\R^n\mid \ell(x)\geq 0 \mbox{ for all } \ell\in Q_{d^*}\cap\R[\ul X]_1\}, \end{align*} is  also the projection of a spectrahedron, as explained above.

For "(i)$\Rightarrow$(ii)" let $\widetilde{S}\subseteq\R^{n+m}$ be a spectrahedron with nonempty interior that projects to $\overline{\conv(S)}$. Let $\sA(\ul X,\ul Y)$ be a strictly feasible linear matrix polynomial defining $\widetilde{S}$. Then consider the quadratic module $Q:=\QM(\sA)$ defined above, and its finite dimensional parts $Q_d:=\QM(\sA)_d$. They fulfill the conditions from (ii), with $d^*=0$.
\end{proof}

{\linespread{1}\bibliographystyle{dpbib} 

}
\end{document}